\documentclass[11pt,leqno]{amsart}

\usepackage[utf8]{inputenc}
\usepackage{amsmath, amssymb, amsthm} 
\usepackage{hyperref}                 
\usepackage{graphicx}
\usepackage[T1]{fontenc}
\usepackage{lmodern}

\newtheorem{theorem}{Theorem}[section]
\newtheorem{lemma}[theorem]{Lemma}

\newtheorem{corollary}[theorem]{Corollary}

\theoremstyle{definition}

\newtheorem{example}[theorem]{Example}
\theoremstyle{remark}
\newtheorem{remark}[theorem]{Remark}

\newcommand{\cpx}[1]{#1 \otimes _{\RR}\CC }

\numberwithin{equation}{section} 

\title[Invariant complex structures on the Euclidean motion group]{Invariant complex structures on the Euclidean motion group as deformations of Calabi-Eckmann structures}

\author{Neda Bokan}
\author{Srdjan Vukmirovi\'{c}}
\address{Faculty of Mathematics, University of Belgrade, Belgrade, Serbia}
\email{srdjan.vukmirovic@matf.bg.ac.rs}

\def\hh{\mathfrak{h}}
\def\gg{\mathfrak{g}}
\def\mm{\mathfrak{m}}
\def\bb{\mathfrak{b}}

\def\RR{\mathbb{R}}
\def\CC{\mathbb{C}}
\def\so3{\mathfrak{so}(3)}
\def\su2{\mathfrak{su}(2)}
\def\e3{\mathfrak{e}(3)}
\def\sl2c{\mathfrak{sl}(2,\CC)}
\def\aff3{\mathfrak{aff}(\RR^3)}


\begin{document}
	
	\begin{abstract}
	We classify, up to equivalence,  all left-invariant complex structures on the three-dimensional Euclidean rigid motion group $SE(3)$. Our approach is based on the correspondence between  complex structures on a real Lie algebra $\e3$ of the motion group and complex Lie subalgebras of the complexification of $\e3$.
		The resulting structures form a one-parameter family $J_\phi$. 
		
		We further establish a connection with the classical Calabi-Eckmann complex structures on $S^3\times S^3.$ More precisely, we construct a one-parameter deformation of Lie brackets connecting $\mathfrak{so}(3)\oplus\mathfrak{so}(3)$ to $\e3,$ for which each endomorphism $J_\phi$ remains integrable. Consequently, every left-invariant complex structure on $SE(3)$ arises as a contraction limit of Calabi–Eckmann structures.
		
		The same family of endomorphisms also defines  complex structures on $\mathfrak{so}(3,1).$
	\end{abstract}
	
	\maketitle

\section*{Introduction}
Lie groups and Lie algebras have played a fundamental role in mathematics for many decades, both as objects of independent study  and as essential tools in various fields, including physics and robotics. In particular, their applications to classical and relativistic mechanics have revealed interesting relationships between different symmetry groups. A well-known example is the relationship between the Galilei and Poincaré groups, reflecting the fact that classical mechanics can be regarded as a limiting case of relativistic mechanics.

These relationships motivated Inönü and Wigner \cite{wigner1953} to introduce the notion of "contraction of Lie groups and Lie algebras", providing a mathematical framework in which one symmetry group can arise as a limiting case of another. More generally, contractions and deformations of Lie algebras offer a natural way to investigate relationships between different Lie algebra structures.

An important problem in this context is the existence and classification of geometric structures on Lie groups and Lie algebras, as well as their behavior under deformations and contractions. Of particular interest are left-invariant complex structures, whose integrability depends on the underlying Lie brackets. 
Samelson showed in \cite{Samelson1953} that any even-dimensional compact real Lie algebra $\gg$ admits an invariant complex structure. D. Snow \cite{Snow1986} classified regular complex structures on reductive  Lie algebras. For four-dimensional solvable Lie algebras, the classification of complex structures is complete: J.E. Snow \cite{Snow1990}  obtained the classification in the case $\dim [\gg,\gg]\leq 2$ and G. Ovando extended this classification in \cite{Ovando} to the case  $\dim [\gg,\gg] = 3.$ When $\dim  \gg=6$ and $\gg$ is nilpotent, S. Salamon determined all such algebras admitting a complex structure and estimated the dimension of the moduli space of complex structures on $\gg$ (cf. \cite{Salamon2000}).

In his papers \cite{Sasaki1981, Sasaki1982}, T. Sasaki used the properties of root-space decomposition of the complexified Lie algebra $\gg^{\CC}$ to classify left-invariant complex structures on Lie groups $SL(3,\RR)$, $GL(2, \RR)$ and $U(2)$. A. Andrada, M.L. Barberis and I. Dotti \cite{Andrada2011} classify all abelian complex structures on six-dimensional Lie algebras.  A. Fino and F. Paradiso \cite{Fino2025} completed the classification of six-dimensional strongly unimodular almost-nilpotent Lie algebras admitting complex structures. T. \v{S}ukilovi\'{c} and S. Vukmirovi\'{c} \cite{SukVuk2023} showed the existence and uniqueness of left-invariant complex structure for cotangent bundle of Heisenberg group  $T^{*} H_{2n+1}$. Inspired by Samelson's construction \cite{Samelson1953} for even-dimensional compact Lie groups, D. N. Pham in \cite{pham2024} has shown that for any compact Lie group $G$, whether of odd or even dimension, the tangent bundle $TG$ admits a left-invariant complex structure, with the natural Lie group structure induced by $G$. More generally, $T^k G$, for $k\geq 1$ admits a left-invariant complex structure, where $T^k G: =T(T^{k-1})$ and $T^0:=G$. 
M.L. Barberis and I. Dotti in \cite{Barberis2004} showed that complex structures on the Euclidean rigid motion group $SE(3)$ and the Lorentzian motion group $E(2,1)$ arise as deformations of a regular complex structure on $SL(2,\CC)$. 

In their celebrated work \cite{calabi1953}, E. Calabi and B. Eckmann constructed invariant, non-K\"ahler complex structures on products of even number of odd-dimensional spheres.   L. Magnin in  \cite{Magnin2010} showed that Calabi-Eckmann structures are the only left-invariant complex structures on $S^3\times S^3$ when regarded as a Lie group.

In this paper  we give a complete classification of left-invariant complex structures on the Euclidean rigid motion group $SE(3).$ 
  Apart from its traditional  importance in geometry, physics and crystallography, the group $SE(3)$ has recently found numerous applications in robotics, kinematics, computer vision, biomechanics and computer graphics.
The classification results in one-parameter family $J_\phi$ of complex structures, described in Theorem \ref{th:glavna}.
Moreover, in Theorem \ref{th:calabi} we describe a close relation of the complex structures $J_\phi$  with the Calabi-Eckmann complex structures on $S^3\times S^3.$
Finaly, in Corollary \ref{cor:so31} we show that those complex  structures naturally extend to Lie algebra $\mathfrak{so}(3,1).$

The paper is organized as follows. In Section \ref{sec:prelims} we introduce the Lie algebra $\e3$  of the Euclidean rigid motion group $SE(3)$ and describe its automorphism group in Lemma  \ref{le:aut}. We also recall the notion of a complex structure on a Lie algebra and present a known example of a complex structure on $\e3$. 

Section \ref{sec:classify} deals with the complete classification of complex structures on the Lie algebra $\e3$. To obtain this classification, we use the fact that the integrability of an almost complex structure $J$ on a Lie algebra $\gg$ is equivalent to the existence of a suitable complex subalgebra $\mm _J \leq \cpx{\gg}$ of its complexification (see Lemma \ref{le:integrable}). 
In particular, complex structures on $\e3$ correspond to certain three-dimensional complex subalgebras $\mm$ of $\cpx{\e3} = \mathfrak{sl}(2,\CC) \ltimes _\rho \CC^3.$ A key step in the classification is Lemma \ref{le:dim2} where we show that $\operatorname{dim}_\CC \bb = 2,$ where $\bb$ is the natural projection of $\mm$ onto $\mathfrak{sl}(2,\CC).$ Hence, $\bb$ is a Borel subalgebra of $\mathfrak{sl}(2,\CC).$ In Theorem \ref{th:glavna} we use this fact to obtain a family $J_\phi,$ $\phi \in (-\frac{\pi}{2}, \frac{\pi}{2})$ of endomorphisms that represent all non-equivalent complex structures on $\e3 .$ At $\phi=0$, this family includes the complex structure described in \cite{Andrada2011}.

In Section \ref{sec:calabi}, Theorem \ref{th:calabi}, we prove that the complex structures $J_\phi$  on $\e3$  arise as deformations of Calabi-Eckmann structures along an Inönü-Wigner contraction  of \(\mathfrak{so}(3)\oplus \mathfrak{so}(3) \) to   \(\mathfrak{e}(3).\)

Moreover, in Corollary \ref{cor:so31} we show that the deformation of Lie brackets also yields a family of complex structures $J_\phi$ on the Lie algebra $\mathfrak{so}(3,1)$. The structure $J_0$ coincides with that found in \cite{Barberis2004}.

\section{Preliminaries}
\label{sec:prelims}

\subsection{Lie algebra $\e3$ of  Euclidean rigid motion group}

Let $\gg $ be a finite-dimensional real Lie algebra with Lie bracket
$[\cdot,\cdot]\colon \gg \times\gg \longrightarrow\gg .$ 
If $\rho\colon\gg \to\mathfrak{gl}(V)$ is a representation, the semidirect product $\gg \ltimes_\rho V$ is the Lie algebra defined on $\gg \oplus V$ by brackets
\begin{equation}
	\label{eq:semi}
[(X,u),(Y,v)]
=
\bigl([X,Y],\rho(X)v-\rho(Y)u\bigr).
\end{equation}
\noindent
The subspace $V = \{ (0,v) \,\mid \, v\in V\}$ is an abelian ideal in $\gg \ltimes_\rho V$.

The tangent Lie algebra of $\gg$ is the semidirect product $T\gg =\gg \ltimes_{\operatorname{ad}}\gg$
where $\operatorname{ad}$ denotes the adjoint representation $\operatorname{ad}_XY = [X,Y].$

The cotangent Lie algebra of $\gg$ is the semidirect product $T^{*}\gg = \gg \ltimes_{\operatorname{ad}^{*}}\gg ^{*},$  defined by means of the coadjoint representation
$ \operatorname{ad}^{*}\colon\gg \longrightarrow
\mathfrak{gl}(\gg ^{*}),$
where 
$$
\langle \operatorname{ad}^{*}_{X}\alpha,Y\rangle
= -\langle\alpha,\operatorname{ad}_XY\rangle = 
-\langle\alpha,[X,Y]\rangle, \enskip \alpha \in \gg^*, \enskip X,Y \in \gg .
$$

The Lie algebra of the affine motion group of $\RR ^n$ is semidirect product $\mathfrak{aff}(\RR^n) = \mathfrak{gl}(n, \RR)\ltimes _\rho \RR ^n$ where 
$$
\rho (A)(v) := Av, \quad A \in \mathfrak{gl}(n, \RR),\, v\in \RR ^n
$$
is the usual multiplication of a square matrix $A$ by a column vector $v$. 

We are  interested in group  of Euclidean rigid motions $SE(3) = SO(3) \ltimes _\rho \RR^3$ of three-dimensional space. Its Lie algebra  $\e3 = \mathfrak{so}(3) \ltimes _\rho \RR ^3$ is a subalgebra of   $\aff3 .$

Since Lie algebra $\mathfrak{so}(3)$ admits an $\operatorname{Ad}$-invariant inner product, it is naturally isomorphic, as an $\mathfrak{so}(3)$-module, to its dual $\mathfrak{so}(3)^*$. Moreover, $\so3$ is isomorphic to $\RR ^3$ via the cross product.
Therefore, we have natural isomorphisms of Lie algebras
\begin{equation}
	T^*\mathfrak{so}(3) \cong T\mathfrak{so}(3) \cong \mathfrak{so}(3)\ltimes _\rho \RR ^3 \cong \e3.
\end{equation}
Throughout the paper, we use the notation $\e3$ and  refer to it as the {\bf Lie algebra of the rigid motion group.}
Brackets \eqref{eq:semi} in the case of algebra $\e3$ become
\begin{equation}
	\label{eq:e3brackets}
	[(A,v), (B,u)]= ([A,B], Au - Bv), \quad A,B \in \mathfrak{so}(3),\enskip u, v\in \RR ^3.
\end{equation}

\noindent
In particular, consider the following basis of the Lie algebra $\mathfrak{so}(3)$ 

$$
E_1=
\begin{pmatrix}
	0&0&0\\
	0&0&-1\\
	0&1&0
\end{pmatrix},
\quad
E_2=
\begin{pmatrix}
	0&0&1\\
	0&0&0\\
	-1&0&0
\end{pmatrix},
\quad
E_3=
\begin{pmatrix}
	0&-1&0\\
	1&0&0\\
	0&0&0
\end{pmatrix}.
$$
\noindent
These matrices satisfy
$$
[E_1,E_2]=E_3,\qquad
[E_2,E_3]=E_1,\qquad
[E_3,E_1]=E_2.
$$
The action of $A_t = t_1 E_1 + t_2 E_2 + t_3 E_3\in \so3 $ on $v\in \RR ^3$ is given by the cross product with the vector $t = (t_1,t_2,t_3),$ i.e. $\rho(A)(v) = A_tv = t\times v.$

\noindent
We use the basis $(e_1,e_2,e_3,e_4,e_5,e_6)$ of Lie algebra $\gg  =  \e3$ defined by
\begin{eqnarray}
	\label{eq:basis}
e_1 &=& (E_1,0), \quad \hspace{2em} e_2 = (E_2,0),\quad   \hspace{2em} e_3= (E_3,0),\\
\nonumber
e_4 &=&  (0,(1,0,0)), \quad   e_5 =  (0,(0,1,0)), \quad  e_6=  (0,(0,0,1))
\end{eqnarray}
with the following nonzero brackets 
\begin{equation}
	\label{eq:komutatori}
\begin{aligned}
	[e_1,e_2]&=e_3,
	&
	[e_2,e_3]&=e_1,
	&
	[e_3,e_1]&=e_2,
	\\
	[e_1,e_5]&=e_6,
	&
	[e_1,e_6]&=-e_5,
	&
	[e_2,e_4]&=-e_6,\\
	[e_2,e_6]&=e_4,
	&
	[e_3,e_4]&=e_5,
	&
	[e_3,e_5]&=-e_4.
\end{aligned}
\end{equation}

\noindent
Since the complexification of Lie algebra $\so3 $ is $\sl2c ,$  the complexification of 
$\e3 $ is 
$$
\cpx{\e3}  =   \sl2c \ltimes _\rho \CC ^3, 
$$ 
where we use the same notation for the representation $\rho$ and its  complexification. It has the same brackets \eqref{eq:komutatori}.
Note that  
$$
\RR ^3 = \{ (0, v)\, \vert \, v\in \RR^3 \}\,\vartriangleleft  \, \e3 \enskip \mbox{and}\enskip \CC ^3 = \{ (0, v)\, \vert \, v\in \CC^3 \}\, \vartriangleleft \, \sl2c \ltimes _\rho \CC ^3
$$ 
are abelian ideals.

Recall that the group of automorphisms of Lie algebra $\gg$ is
$$
\operatorname{Aut}(\gg) = \{ F\in \mathfrak{gl}(\gg) \, \mid \, F([A,B]) = [F(A), F(B)] \enskip \mbox{for all}\enskip A, B \in \gg  \}.
$$

\begin{lemma}
	\label{le:aut}
	Every automorphism of the Lie algebra $\e3=\mathfrak{so}(3)\ltimes _\rho \mathbb R^3$ with brackets \eqref{eq:e3brackets} is a composition
	$$
	\Phi_{w,C,\lambda} = \Theta _\omega \circ R_C \circ \xi _\lambda,
	$$
	where
	\begin{align*}
		R_C(A,v) &= (CAC^{-1}, Cv), \enskip C\in SO(3),\\
		\Theta_\omega (A,v) &= (A, v-A\omega), \enskip \omega  = (\omega_1, \omega _2, \omega _3)\in \RR ^3,\\
		\xi _\lambda  (A,v) &= (A, \lambda v), \enskip \lambda \in \RR, \, \lambda \neq 0.
	\end{align*}
\end{lemma}
\begin{proof}
	The ideal $\RR ^3 \vartriangleleft \e3$ is  invariant under every automorphism. Hence any
	\(\Phi\in\operatorname{Aut}(\mathfrak{se}(3))\) has the form
	\[
	\Phi(A,v)=\bigl(\alpha(A),\beta(A)+Mv\bigr),
	\]
	where
	\[
	\alpha:\mathfrak{so}(3)\to\mathfrak{so}(3),\qquad
	\beta:\mathfrak{so}(3)\to\mathbb R^3,\qquad
	M:\mathbb R^3\to\mathbb R^3
	\]
	are linear, with \(\alpha\) and \(M\) invertible.
	
	Since
	$[(A,0),(B,0)]=([A,B],0),$
	comparison of the first components gives
	$\alpha([A,B])=[\alpha(A),\alpha(B)]. $
	Thus \(\alpha\) is an automorphism of \(\mathfrak{so}(3)\). Every automorphism of
	\(\mathfrak{so}(3)\) is inner, so there exists \(C\in SO(3)\) such that
	\begin{equation}
		\label{eq:alpha}
		\alpha(A)=CAC^{-1}.
	\end{equation}
	
	Next, by aplying $\Phi$ to the relation
	$[(A,0),(0,v)]=(0,Av),$
	we obtain
	$$
	MAv=\alpha(A)\, Mv
	=CAC^{-1}Mv.
	$$
	Denote $D=C^{-1}M,$ so we obtain $D Av =A Dv $
	for every \(A\in\mathfrak{so}(3)\) and \(v\in\mathbb R^3\). Thus \(D\) commutes with every skew-symmetric \(3\times3\) matrix. An elementary matrix computation shows that the only such matrices are scalar matrices $D = \lambda E,$ i.e.
	$$
	M=\lambda C,\enskip \lambda\ne0. 
	$$
	
	To determine \(\beta\), from
	$\Phi([(A,0),(B,0)]) = [\Phi(A,0),\Phi(B,0)]$
	we obtain
	\begin{equation}
		\label{eq:beta1}
		\beta([A,B]) = \alpha(A)\beta(B)-\alpha(B)\beta(A).
	\end{equation}
	
	\noindent
	Define
	$\gamma(A)=\beta(C^{-1}AC)$ and 
	\eqref{eq:beta1} becomes
	\begin{equation}
		\label{eq:cocycle}
		\gamma([A,B])
		=
		A\gamma(B)-B\gamma(A),
		\enskip A,B\in\mathfrak{so}(3).
	\end{equation}
	
	\noindent
	One can check directly  that from \eqref{eq:cocycle}  it follows the existence of  $w\in \RR ^3$ such that 
	$\gamma(A)=-Aw$
	for every \(A\in\mathfrak{so}(3)\). This also may follow from vanishing of the first cohomology of $\mathfrak{so}(3)$ and the Whitehead's first lemma.
	We obtain
	\[
	\beta(A)=-(CAC^{-1})w.
	\]
	
	\noindent
	Consequently,
	\[
	\Phi(A,v)
	=
	\left(
	CAC^{-1},
	\lambda Cv-(CAC^{-1})w
	\right),
	\]
	so every $\Phi \in Aut (\e3)$ is a composition of $\Theta _\omega,$ $R_C,$ and $\xi _\lambda,$ defined in the statement of the lemma.
	Conversely, a direct computation shows that $\Theta _\omega,$ $R_C,$ and $\xi _\lambda$ are automorphisms of $\e3 .$
\end{proof}

\subsection{Complex structures on Lie algebras}

Let $G$ be a real Lie group of even dimension $2n$ and $\gg $ its  Lie algebra. An almost complex structure on $G$ is an endomorphism $J:TG\to TG$ of the tangent bundle satisfying  
$$
J^{2}=-\operatorname{Id}.
$$
From the Newlander-Nirenberg  theorem $J$ arises from a complex structure on $G$ if and only if  Nijenhuis tensor
$$
N_J(X,Y) :=[JX,JY]-[X,Y]-J[JX,Y]-J[X,JY],
$$
vanishes for all vector fields $X,Y \in \mathfrak{X} (G).$ An almost complex structure is left-invariant if it commutes with differentials of left translations on $G.$

Therefore left-invariant complex structure on Lie group $G$ is described by complex structure on its Lie algebra $\gg,$ i.e. endomorphism   
$J\colon\gg \longrightarrow\gg $
satisfying

$$
J^2=-\operatorname{Id}_{\gg }
\qquad\text{and}\qquad
N_J(X,Y)=0
$$
for all $X,Y\in\gg $.

Equivalently, one may consider the complexification $\cpx{\gg}$  of the Lie algebra $\gg .$
Almost complex structure $J$ extends complex-linearly to $\cpx{\gg}$. Its $i$-eigenspace is
\begin{equation}
	\label{eq:mj}
\mm _J
=
\{X-iJX \,\vert\, X\in\gg \}.
\end{equation}

This is a complex vector subspace of $\gg _{\CC }$ of complex dimension $n$. The $-i$-eigenspace is its complex conjugate
\begin{equation}
	\label{eq:barmj}
\overline{\mm }_J
=
\{X+iJX \,\vert\, X\in\gg \}.
\end{equation}
We see that  $\gg _{\CC }=\mm _J\oplus\overline{\mm }_J$ and in particular $\mm _J\cap\overline{\mm }_J=\{0\}.$ 

\noindent
The proof of the following, well known fact, can be found in \cite{Nomitzu1969}.
\begin{lemma}
	\label{le:integrable}
Complex  structures  on Lie algebra $\gg $ are in one-to-one corespodence with complex subalgebras  $\mm \leq \cpx{\gg} $
such that $\cpx{\gg}=\mm \oplus\overline{\mm }.$
\end{lemma}

Consequently, the classification of left-invariant complex  structures on a real Lie algebra $\gg $ can be reformulated as the classification of complex Lie subalgebras $\mm \leq \gg _{\CC }$
\noindent
such that
$$
\dim_{\CC }\mm =\frac12\dim_{\RR }\gg ,
\qquad
\gg _{\CC }
=
\mm \oplus\overline{\mm }.
$$
From a basis of the subalgebra $\mm$ and its conjugate $\overline{\mm}$ the structure $J$ is easily reconstructed.

\begin{example}
	\label{ex:complex}
	In \cite{ovando2012} it was shown that $\e3$ admits a complex  structure that is in basis \eqref{eq:komutatori} given by
	\begin{equation}
		\label{eq:Jovando}
	Je_1 = -e_2, \, Je_2 = e_1,\, Je_3 = e_6,\, Je_6 = -e_3, Je_4 = -e_5,\, Je_5 = e_4.
	\end{equation}
	The corresponding subalgebra is
	$$
	\mm _J = \operatorname{span}_\CC (f_1 = e_1 + ie_2,\, f_2 = e_5 - ie_4,\, f_3 = e_6 + i e_3).
	$$
	The commutators 
	$$
	[f_1, f_2]=0, \quad [f_1, f_3] = -f_1 - f_2, \quad [f_2, f_3]=-f_2
	$$
	belong to $\mm$, so it is a complex subalgebra. 
\end{example}

\noindent
Two complex structures $J_1, J_2$ on Lie algebra $\gg$ are regarded as equivalent if
\begin{itemize}
	\item $J_1 = F^{-1}J_2F$,  $F\in \operatorname{Aut} (\gg ),$ where all matrices are represented in the same basis of algebra $\gg ,$
	
	\centerline{or equivalently,}
	\item corresponding subalgebras $\mm _{J_1}, \mm _{J_2}\leq \cpx{\gg} $ are related by the complexification $F_\CC : \cpx{\gg} \to \cpx{\gg}$ of a real  automorphism $F\in Aut (\gg ),$ i.e. $F_\CC (\mm _{J_1})  = \mm _{J_2}.$
	
\end{itemize}

	\section{Classification of complex structures on Lie algebra $\e3$}
	\label{sec:classify}
	
Classification of complex structures on  $\e3$ is according to Lemma \ref{le:integrable} equivalent to classification of $3$-dimensional complex subalgebras 
$$
\mm  \leq \cpx{\e3}   = \sl2c \ltimes _\rho  \CC ^3, \quad\mm \cap \overline{\mm } = \{ 0\}.
$$ 
Note that natural projection
$$
\pi : \sl2c \ltimes _\rho \CC ^3 \to \sl2c , \quad \pi(A,v):= A,
$$	
is epimorphism of Lie algebras with kernel $\operatorname{Ker}(\pi) = \CC ^3$ and image $\operatorname{Im}(\pi) = sl(2, \CC ).$

\noindent
Denote by $\hh  := \mm  \cap \CC ^3 $ and $\bb := \pi (\mm )$. 
\begin{lemma}
	\label{le:dimensions}
 $\bb  \leq \sl2c$ is subalgebra  and $\hh \vartriangleleft \mm $ is ideal.
\end{lemma}
\begin{proof}
	The first statement follows since $\bb$ is image of subalgebra $\mm$ under homomorphism.	
For the second statement
\begin{align*}
 [\mm , \hh] &= \{ [(A, v), (0, u)]\, \mid \,   v,u \in \hh , \, A \in \bb \} = \{ (0, Au)\}\subset \CC ^3,\\
 [\mm , \hh] &\subset  [\mm, \mm] \subset  \mm 
\end{align*}
and therefore $[\mm , \hh] \subset \mm \cap \CC^3 = \hh .$
\end{proof}

In the Example \ref{ex:complex} we see that  $\bb = \operatorname{span} _\CC (e_1 + ie_2, ie_3)\leq \sl2c $ and $\operatorname{dim}_\CC \bb = 2.$ Also $\hh =  \operatorname{span} _\CC (-ie_4 + e_5)$ and  $\operatorname{dim}_\CC \hh = 1.$ Now we show that this must hold in general.

\begin{lemma}
	\label{le:dim2}
	If $\mm \leq \cpx{\e3}  = \sl2c \ltimes _\rho \CC^3$ is $3$-dimensional complex subalgebra that corresponds to a complex structure on $\gg = \e3,$ then $\operatorname{dim}_\CC \bb = 2.$
\end{lemma}
	\begin{proof}
	i)	If $\operatorname{dim}_\CC \bb = 0$ then $\mm = \operatorname{Ker} (\pi) = \CC ^3,$ but since $\overline{\mm} = \CC^3$ and  $\mm \cap \overline{\mm} = \CC ^3 \neq \{ 0\}$ the subalgebra $\mm$ can't correspond to a complex structure.
		
	ii) Suppose that 	 $\operatorname{dim}_\CC \bb = 3.$ Then $\bb = \sl2c $ and we have well defined map $L: \sl2c  \to \CC^3$ such that 
	$$
	\mm = \{(A, L(A))\, \mid \, A\in \sl2c \}.
	$$
	From linearity of $\pi$ it follows that $L$ is linear. Moreover
	$$
	[(A, L(A)),(B, L(B)) ] = ([A,B], AL(B)-BL(A)) \in \mm,
	$$
	and therefore $L([A,B]) = AL(B) - BL(A)$ so $L$ is a derivation, i.e $L$ is $1$-cocycle.  Since $\sl2c $ is semisimple, from the Whitehead's first lemma it follows that the first cohomology $H^1(\sl2c )$ is trivial. Therefore, the derivation $L$ is inner, that is, there exists $v\in \CC ^3$ such that $L(A) = Av, A\in \sl2c .$ We conclude that algebra $\mm$ and its conjugate are  described by $v\in \CC ^3:$ 
	\begin{align*}
		\mm  = \mm _v & = \{ (A, Av)\, \mid \, A \in \sl2c \},\\
		\overline{\mm} = \mm _{\overline{v}} & = \{ (A, A\overline{v})\, \mid \, A \in \sl2c \}.
	\end{align*} 
Let  $x\in \mm \cap \overline{\mm}.$ The condition $x\neq 0$  is equivalent to existence of  $A\neq 0$ such that 
$$
(A, Av) = x = (A, A\overline{v}) \quad \Leftrightarrow\quad A(v-\overline{v}) = 0.
$$
But, for each $v\in \CC^3$, or equivalently for each subalgebra $\mm = \mm _v,$  there exists $A\in \sl2c \cong so(3,\CC)$ such that $A(v-\overline{v}) = 0,$ i.e. there exists $0\neq x = (A, Av)\in \mm\cap \overline{\mm}.$ Therefore, $\mm$ can't correspond to a complex structure. 

iii) Suppose that 	 $\operatorname{dim}_\CC (\bb) = 1,$ i.e. $\bb = \CC A, \, A\in \sl2c .$ Since $\hh \subset \CC ^3$ is ideal in $\mm,$ it is invariant under the action of  $A\in \sl2c \cong so(3,\CC).$ From  $\mm = \CC (A,0) \oplus \hh$ 
we obtain  $\overline{\mm} = \CC (\overline{A},0) \oplus \overline{\hh}.$ Since $\operatorname{dim}_\CC \hh = 2 = \operatorname{dim}_\CC \overline{\hh}$ and  $\hh, \overline{\hh} \subset \CC ^3$ we conclude that $\operatorname{dim}_\CC \hh \cap \overline{\hh} \geq 1,$ so there exists $v \in \hh \cap \overline{\hh},\, v\neq 0$ and therefore $v\in \mm \cap \overline{\mm}.$ Therefore, $\mm$ can't correspond to a complex structure. 
	\end{proof}
	
\begin{theorem}
	\label{th:glavna}
	Let \eqref{eq:basis} be a basis of $\e3$ with brackets \eqref{eq:komutatori}.
	
	\noindent
a) Non-isomorphic subalgebras of $\cpx{\e3} $ associated to  complex  structure are
\begin{equation}
	\label{eq:mphi}
	\mm _\phi = \operatorname{span}_\CC (e_1 + i e_2, -ie_4+ e_5,  i e_3 + e^{i\phi} e_6), \enskip \phi \in (-\frac{\pi}{2}, \frac{\pi}{2}). 
\end{equation}

\noindent
b) The corresponding complex structures are represented by matrices
	\begin{equation}
		\label{eq:jphi}
	J_\phi  = \begin{pmatrix}
		0 & 1 & 0& 0& 0& 0\\
		-1& 0& 0& 0& 0& 0\\
		0& 0& \frac{\sin \phi}{\cos \phi}& 0& 0& -\frac{1}{\cos \phi}\\
		0& 0& 0& 0& 1& 0\\
		0& 0& 0& -1& 0& 0\\
		0& 0& \frac{1}{\cos \phi}& 0& 0& -\frac{\sin \phi}{\cos \phi}
	\end{pmatrix}, \quad \phi\in (-\frac{\pi}{2}, \frac{\pi}{2}).
\end{equation}

\end{theorem}	
\begin{proof}
	According to Lemma \ref{le:dimensions} we can suppose that $\operatorname{dim} _\CC \bb  = 2.$ Up to inner automorphisms of $\sl2c $ the unique $2$-dimensional subalgebra $\bb\leq \sl2c $ is the Borel subalgebra, that we now describe. 

\noindent	
	Since $\sl2c \cong so(3,\CC)$ we take basis 
	\begin{equation}
H = \begin{pmatrix}
	0 & -i & 0\\
	i & 0 & 0\\
	0 & 0& 0
	\end{pmatrix}, \quad
E = \begin{pmatrix}
	0 & 0 & i\\
	0 & 0 & -1\\
	-i & 1& 0
\end{pmatrix}, \quad
F = \begin{pmatrix}
	0 & 0 & i\\
	0 & 0 & 1\\
	-i & -1& 0
\end{pmatrix},
	\end{equation}	
that spans  $\sl2c   = \operatorname{span}_\CC (H,E,F).$ 
In this basis the brackets of $\sl2c $ are
\begin{equation}
	[H,F] = -F, \quad [H, E] = E, \quad [E,F ] = 2 H.
\end{equation}
We take the Borel subalgebra
$$
\bb = \operatorname{span}_\CC (H,E).
$$

The algebra $\hh = \mm \cap \CC$ is $1$-dimensional so we suppose $\hh = \CC v',$ $v' = (0, v_+).$ Since $\hh$ is ideal in $\mm$, it is preserved under the action of $\bb,$ so we calculate  $v_+$ as common eigenvector of  $H, E$ and obtain $v_+ = (-i,1,0).$ 
Moreover, we take basis $(v_+, v_-, v_0)$ of $\CC ^3$ with $v_- = (i,1,0),$ $v_0 = (0,0,1).$ The action of $\sl2c $ on that basis is:
\begin{align}
	\nonumber
	Hv_+ & = v_+, \quad \quad 	Hv_- = -v_-, \quad 	Hv_0 = 0, \\
		\label{eq:dejstvo}
		Ev_+ &= 0, \quad \quad \enskip	Ev_- = 2 v_0, \quad 	Ev_0 = -v_+, \\
		\nonumber
			Fv_+ &= -2v_0, \quad 	Fv_- = 0, \quad \quad 	Fv_0 = v_-. 
\end{align}
Subalgebra $\mm$ is spanned by 
$$
\mm = \operatorname{span}_\CC ( v', H', E' )
$$
where 
$$
 H' = (H, v_H), \quad E' = (E, v_E), 
$$ 
for some $v_H, v_E \in \CC ^3,$ are  lifts of $H$ and $E.$ They are subject to the condition that $\mm$ is subalgebra:
\begin{align*}
	[H', v'] &= [(H, v_H), (0, v_+)] = (0, Hv_+) = (0, v_+) = v' \in \mm,\\
	[E', v'] &= [(E, v_E), (0, v_+)] = (0, Ev_+) = (0, 0) = 0 \in \mm,\\
	[H', E'] &= [(H, v_H), (E, v_E)] = (E, Hv_E-Ev_H)\in \mm.
\end{align*}
The first two conditions are fulfilled and from the third, for some $\alpha, \beta , \gamma \in \CC$  
$$
[H', E'] = \alpha v' + \beta H' + \gamma E' = (\beta H + \gamma E, \alpha v_+ + \beta v_H + \gamma v_E).
$$
But $\pi( [H', E']) = [\pi (H'), \pi (E')] = [H, E] = E,$ so $\beta H + \gamma E = E$ and therefore $\beta = 0, \gamma = 1.$
Additionaly, the condition
\begin{equation}
	\label{eq:cond}
	Hv_E - Ev_H = \alpha v_+ + v_E
\end{equation}
has to be fulfilled. Now, for some complex coefficients we write 
$$
v_H = x\,v_+ + y\, v_- + z\,v_0, \quad v_E = p \,v_+ + q\, v_- + r\,v_0,
$$
and from  \eqref{eq:dejstvo}, \eqref{eq:cond} we obtain
\begin{equation}
v_H = x \,v_+ + y \,v_- + \alpha\, v_0, \quad v_E = p\,v_+ - 2 y \,v_0,
\end{equation}
for $x,y,\alpha, p \in \CC.$

We can change basis of $\mm$ using the automorphism $H'\to H' - x v',$ $E'\to E' - p v',$ $v'\to v'$ to further simplify $v_H$ and $v_E,$ or equivalently set $x =0=p.$
Therefore $\mm = \operatorname{span}_\CC (v', H', E'),$
$$
 v' = (0, v_+),\enskip  H' = (H, y v_- + \alpha v_0), \enskip E' = (E, -2y v_0).
$$
Now we use complexifications of automorphisms of $\e3$ from Lemma \ref{le:aut} to further simplify $\mm,$ that keeps representing the equivalent complex structure.

For vector $\omega  = (\omega _1, \omega _2, \omega _3)\in \RR ^3$ we use automorphism $\Theta _\omega (A,v) = (A, v- A\omega).$
A direct computation gives
\begin{align*}
H\omega  &= \frac{1}{2}(i\omega _1 + \omega _2) v_+ + \frac{1}{2}(i\omega _1 - \omega _2) v_-, \\
E\omega & = -\omega _3 v_+ + (-i\omega _1 + \omega _2) v_0.
\end{align*}
If we choose real vector $\omega$ such that $i\omega _1 - \omega _2 = 2 y$, and $\omega _3 \in \RR$ arbitrary, modulo $v'$ (the component that  already belongs to $\mm$), we obtain a simpler basis of $\mm:$
\begin{align*}
H'_\alpha &= \Theta _\omega (H, y v_- + \alpha v_0) = (H, \alpha v_0),\enskip \alpha \in \CC \\
E' &= \Theta _\omega (E, -2y v_0) = (E,0).
\end{align*}
Finally, we can use automorphism $\xi _\lambda  (A,v) = (A, \lambda v),$ $\lambda \in \RR, \lambda \neq 0$  to achieve $\alpha \to \lambda \alpha,$ i.e. we can suppose that $\alpha  = e^{i\phi}, \phi \in (-\frac{\pi}{2}, \frac{\pi}{2}] .$  

\noindent
The resulting subalgebra is $\mm _\phi= \operatorname{span}_\CC ( v', H_\phi ', E'),$
\begin{equation}
\label{eq:mfi} v' = (0, v_+), \enskip H_\phi' = (H, e^{i\phi}v_0), \enskip E' = (E, 0),\quad \phi \in (-\frac{\pi}{2}, \frac{\pi}{2}].
\end{equation}

We have already used two types of automorphisms of $\so3$ to simplify parameters describing subalgebra $\mm .$ It is easy to check that automorphism of the third type,  $R _C (A,v) = (CAC^{-1}, Cv), \, C\in SO(3),$ cannot change parameter $\phi.$ Therefore, all subalgebras $\mm _\phi$ are mutually non-equivalent.

It remains to check the condition $\mm _\phi \cap \overline{\mm}_\phi  = \{ 0\},$ where
\begin{equation}
\label{eq:barmfi}
\overline{\mm}_\phi  = \operatorname{span}_\CC ((0, v_-), (-H, e^{-i\phi}v_0), (-F,0)).
\end{equation}
If $x \in \mm _\phi \cap \overline{\mm}_\phi ,$ for some complex coefficients,  we obtain
$$
p (0, v_+) + q (H, e^{i \phi}v_0) + r (E,0) = x  =
p' (0, v_-) + q' (-H, e^{-i \phi}v_0) + r' (-F,0).
$$
This is equivalent to 
$$
qH + rE = -q'H - r'F, \quad pv_+ + q e^{i\phi} v_0 = p' v_- + q'e^{-i\phi} v_0.
$$
We use that $(H,E,F)$ is basis of $\sl2c $ and $(v_+, v_-, v_0)$ is basis of $\CC ^3$ to obtain nontrivial solution $x\neq 0$ if and only if $e^{i\phi} + e^{-i\phi} =0,$ i.e.  $\cos \phi =0.$ In the case $\cos \phi = 0$ we have $\mm _\phi \cap \overline{\mm}_\phi  \neq  \{ 0\}$ and $\mm _\phi$ doesn't correspond to a complex structure.
This rules out parameter $\phi = \frac{\pi}{2}.$ 

Let us  find subalgebra $\mm$ and the associated complex structure in basis \eqref{eq:basis}. Note that
$$
H = i E_3, \quad E = E_1 + iE_2, \quad F = -E_1 + iE_2,
$$
$$
(0,v_+) = -i e_4 + e_5, \quad (0,v_-) = i e_4 + e_5, \quad (0, v_0) = e_6.
$$
From \eqref{eq:mfi} and \eqref{eq:barmfi} we conclude that 
\begin{align*}
	\mm _\phi &= \operatorname{span}_\CC (-ie_4 + e_5, ie_3 + e^{i\phi}e_6, e_1 + i e_2 ),\\
	\overline{\mm} _\phi &= \operatorname{span}_\CC (ie_4 + e_5, -ie_3 + e^{-i\phi}e_6, e_1 - i e_2 )
	\end{align*}
and we obtain the statement a). 
Statement b) follows directly from a) and fact that $\mm _\phi$ is $i$-eigenspace of $J_\phi$ (see formula \eqref{eq:mj}). 
\end{proof} 

\begin{remark}
	For $\phi = 0$ we obtain complex structure from Example \ref{ex:complex}.
\end{remark}

\begin{remark}
In \cite{smolencev2025} the authors obtain four families of complex structures, but they didn't check their equivalence.   Since those four families are far from complete classification we check only $J_4$ for $\Psi_{44} = 0, \Psi_{41} = 1.$
After adjusting the base of $\e3$ by permuting $e_1 \leftrightarrow e_4, e_2 \leftrightarrow e_5, e_3 \leftrightarrow e_6$ to match commutators \eqref{eq:komutatori}, the structure is
$$
J_4e_4 = e_1, \,J_4e_1 = -e_4, \, J_4e_5 = e_6, \, J_4e_6 = -e_5,  \,J_4e_2 = e_3, \, J_4e_3 = -e_2.
$$
But  after the automorphism of $\so3$
$$
F:e_1,e_2,e_3,e_4,e_5,e_6 \mapsto e_2,e_3,e_1,e_5,e_6,e_4
$$
this is the structure $-J_\phi,$ for $\phi =0,$ from our classification.
\end{remark}

\section{Complex structures on $\e3$ as a contraction limits of Calabi-Eckmann structures}
\label{sec:calabi}

It is well known (see \cite{Nomitzu1969, Magnin2010}) that all non-eqivalent left-invariant complex structures on direct product of spheres  $S^3\times S^3,$ regarded as a Lie group, are so called Calabi-Eckmann structures.
Namely, each copy of $S^3$ is considered as Hopf bundle $S^1\to S^3 \to S^2.$ Both copies of base spaces  $S^2$ carry the natural and unique complex structure. The two-dimensional family of Calabi-Eckman structures is defined to "permute" the fibers, i.e. it is parameterized on two-dimensional torus $S^1\times S^1.$
Since, the complex structures are left-invariant, they can be described as complex structures on corresponding Lie algebra $\su2 \oplus  \su2 .$

In this section we show that the complex structures on $\e3$ from our classification are closely related to the Calabi-Eckmann complex structures on $\su2 \oplus  \su2  = \so3 \oplus  so(3).$

Let $(X_1, X_2, X_3, Y_1, Y_2, Y_3)$ be a standard basis of  $\su2 \oplus  \su2 ,$ 
\begin{align}
[X_1, X_2] &= X_3, \enskip [X_2, X_3] = X_1, \enskip [X_3,
 X_1] = X_2, \nonumber \\
 [Y_1, Y_2] &= Y_3, \enskip [Y_2, Y_3] = Y_1, \enskip [Y_3, Y_1] = Y_2, \label{eq:su2su2}\\
 [X_i, Y_j]& = 0, \enskip i,j= 1,2,3.\nonumber
\end{align}

\noindent
We also define  family of Lie algebras by the nonzero brackets: 
\begin{align}
[E_1, E_2] &= E_3, \enskip [E_2, E_3] = E_1, \enskip [E_3,E_1] = E_2, \nonumber\\
\gg _s:\quad [F_1, F_2] &= s F_3, \enskip [F_2, F_3] = s F_1, \enskip [F_3, F_1] = s F_2, \quad s\in \RR, \label{eq:komEiFi} \\
 [E_1, F_2]& = F_3, \enskip  [E_1, F_3] = -F_2, \enskip [E_2, F_1]= - F_3,  \nonumber\\
 [E_2, F_3]& =  F_1, \enskip [E_3, F_1] = F_2, \enskip  [E_3, F_2] = -F_1. \nonumber
\end{align}
There exists the following isomorphism of Lie algebras
$$
\mathfrak{g}_s \cong
\begin{cases}
 \mathfrak{su}(2)\oplus \mathfrak{su}(2), & s>0,\\[2mm]
	\e3, & s=0,\\[2mm]
	\mathfrak{so}(3,1), & s<0,
\end{cases}
$$
(see \cite{Morales2025}), where parameter $s$ corresponds to the cosmological constant in Euclidean 3D gravity.

In order to deform $\su2 \oplus  \su2 $ to $\e3 = \so3 \ltimes _\rho \RR ^3= \su2 \ltimes _\rho \RR ^3$ let us introduce (see \cite{wigner1953}) the "deformation basis":
\begin{equation}
	\label{eq:EiFi}
	E_i = X_i + Y_i, \quad F_i = t (X_i - Y_i), \enskip t \geq 0. 
\end{equation}
Starting from \eqref{eq:su2su2} the brackets of "deformation basis" \eqref{eq:EiFi} produce Lie algebra $\gg _s$ for $s = t^2,\, t\geq 0,$ so we regard $\gg _{t^2}$  as a "deformation" family from $ \mathfrak{su}(2)\oplus \mathfrak{su}(2)$ to $\e3.$

For $t  = 0$  we use the identification 
\begin{equation}
	\label{eq:identification}
(E_1, E_2, E_3, F_1, F_2, F_3) \leftrightarrow (e_1,e_2,e_3,e_4,e_5,e_6)
\end{equation}
of basis of the $\gg_0$ and $\e3$, both with the same brackets
\eqref{eq:komutatori}.

The Calabi-Eckmann structures $J_\tau$ on $\su2 \oplus  \su2 $ is (see \cite{Nomitzu1969} or \cite{Magnin2010}, Corollary 3. for details) correspond to subalgebras $\mm _\tau \leq \cpx{(\su2 \oplus \su2 )}$ given by
\begin{equation}
	\label{eq:mtau}
\mm _{\tau} = \operatorname{span}_\CC (X_1 + i X_2, Y_1 + i Y_2, X_3 + \tau Y_3), \enskip \tau \in \CC, \, \operatorname{Im} (\tau )\neq 0. 
\end{equation}

The condition $\operatorname{Im} (\tau )\neq 0$ ensures that $\mm_\tau \cap \overline{\mm}_\tau = \{0\}.$

Now we show that each complex structure $J_\phi$ on $\e3$  is  a contraction limit of Calabi-Eckmann complex structures.

\begin{theorem}
	\label{th:calabi}
a) Each  complex structure \(J_\phi\) given by \eqref{eq:jphi} is also complex structure of  the entire family of Lie algebras $\gg _{s},$ $s\in \RR,$ with Lie brackets given by \eqref{eq:komEiFi}.

b) For each fixed $\phi$ and every $t>0,$ the pair $(\gg _{t^2},J_\phi)$ is isomorphic to $\mathfrak{so}(3)\oplus\mathfrak{so}(3)$ endowed with a Calabi–Eckmann complex structure $J_\tau,$ where
\begin{equation}
	\tau=\tau_\phi (t)
	=\frac{1+it e^{i\phi}}{1-it e^{i\phi}}.
\end{equation}
As $t \to 0^+$, the Lie algebra $\mathfrak g_{t^2}$ contracts to $\e3$, and the family of pairs $(\gg_{t^2}, J_\phi)$ contracts to $(\e3 ,J_\phi).$
\end{theorem}
\begin{proof} a) Recall that  family of  structures $J_\phi, \, \phi \in (-\frac{\pi}{2},\frac{\pi}{2})$  correspond to the family of complex subalgebras  $\mm _\phi \leq \cpx{\e3}$  given by \eqref{eq:mphi}.
	
	Using identification \eqref{eq:identification} we obtain that the same endomorphisms $J_\phi$ on $\gg _s$  correspond to a subspace  
	\begin{equation}
		\label{eq:tildem}
	 \tilde \mm _\phi := \operatorname{span}_\CC (E_1 + i E_2, -iF_1+ F_2, i E_3 + e^{i\phi} F_3 ) \subset \cpx{\gg_s}. 
	\end{equation}
	It is straightforward to prove that  $\tilde \mm _\phi$ is complex subalgebra of $\gg _s$ and hence  $J_\phi$ are complex structures on $\gg _s,\, s\in \RR.$
	
	b) To contract Calabi-Eckmann structure $J_\tau$ to $J_\phi$ we compare corresponding subalgebras $\mm_\tau$ and $\tilde \mm _\phi$ given by \eqref{eq:mtau} and \eqref{eq:tildem}. Since
\begin{align}
	\label{eq:veza}
 E_1 + i E_2 &= X_1 +  Y_1 + i(X_2 +  Y_2) = (X_1 + i X_2) +  (Y_1 + i Y_2),\\
 \nonumber
 F_1 + i F_2 &= t   (X_1 -  Y_1) + it   (X_2 -  Y_2) = t  (X_1 + i X_2) - t  (Y_1 + i Y_2),
\end{align}
we see the subalgebras $\mm _\tau$ and $\tilde \mm _\phi$ are the same, i.e. determine the same complex structure if and only if
$$
i E_3 + e^{i\phi} F_3 = \mu (X_3 + \tau Y _3),
$$
for some $\mu \in \CC, \,\mu \neq 0.$ Using \eqref{eq:EiFi}, we eliminate $\mu$ to obtain
\begin{equation}
	\label{eq:taut}
\tau  = \tau _\phi (t ) := \frac{1+ie^{i\phi}t }{1-ie^{i\phi}t }.
\end{equation}
\noindent
For \(\tau=\tau_\phi(t)\), equations \eqref{eq:mtau} and \eqref{eq:veza} show that the complex subalgebra \(\mathfrak m_\tau\) defining the Calabi-Eckmann structure \(J_\tau\) coincides with the complex subalgebra $\tilde \mm _\phi$ defining \(J_\phi\). Since a complex structure is uniquely determined by its \(i\)-eigenspace in the complexified Lie algebra, it follows immediately that
$J_{\tau_\phi(t)}=J_\phi.$

Thus, although the Calabi-Eckmann parameter \(\tau_\phi(t)\) depends on \(t\), the corresponding complex structure, expressed on the fixed underlying vector space in the basis \((E_1,E_2,E_3,F_1,F_2,F_3)\), is independent of \(t\) and is precisely the structure \(J_\phi\) given by \eqref{eq:jphi}.
\end{proof}

\begin{corollary} 
	\label{cor:so31}
	Structures  \(J_\phi\) given by \eqref{eq:jphi} are complex structures on the Lie algebra $\mathfrak{so}(3,1)$. There exists a complex structure $J_\RR$  on  $\mathfrak{so}(3,1)$ that does not belong to the family $J_\phi.$
	\end{corollary}
\begin{proof}
	In Theorem \ref{th:calabi} a) we proved that the structures $J_\phi$ are complex structures on Lie algebras $\gg _s$ for $s<0$, which is isomorphic to $so(3,1).$ 
	
	Since algebra $so(3,1)  = \sl2c _\RR$ is realification of complex Lie algebra $\sl2c $  it carries the natural complex structure $J_\RR$ obtained as multiplication by $i.$
Therefore,  $J_\RR$ is bi-invariant, i.e. $J_\RR [X, Y] = [J_\RR X, Y]$ holds for all $X,Y \in so(3,1).$

For any complex structure $J_\phi, \, \phi \in (-\frac{\pi}{2}, -\frac{\pi}{2})$ take $X = E_1,$ $Y = E_2.$ Then 
\begin{align*}
J_\phi [X,Y] &= J_\phi [E_1, E_2] = J_\phi E_3 = \frac{\sin \phi}{\cos \phi}E_3 + \frac{1}{\cos \phi} F_3,\\
[J_\phi X, Y] &= [J_\phi E_1, E_2] = [-E_2, E_2] = 0,
\end{align*}
that is $J_\phi [X,Y]\neq [J_\phi X, Y]$ and therefore $J_\phi$ is not bi-invariant. Since $J_\RR$ is bi-invariant and $J_\phi,  \, \phi \in (-\frac{\pi}{2}, \frac{\pi}{2})$ are not bi-invariant, they are not equivalent.
\end{proof}	

\begin{remark}
In paper \cite{Barberis2004} 
the authors showed that a regular invariant complex structure on $\sl2c\cong \mathfrak{so}(3,1)$ remains integrable under the classical deformation to the Euclidean motion algebra $\e3 .$ It is easy to check that the structure they found on  algebra $\gg _s, s<0$ is exactly our $J_0,$ i.e. obtained for $\phi = 0.$   
\end{remark}

\begin{figure}[h!]
	\centering
	\IfFileExists{krugovi.png}{\includegraphics[width=0.8\textwidth]{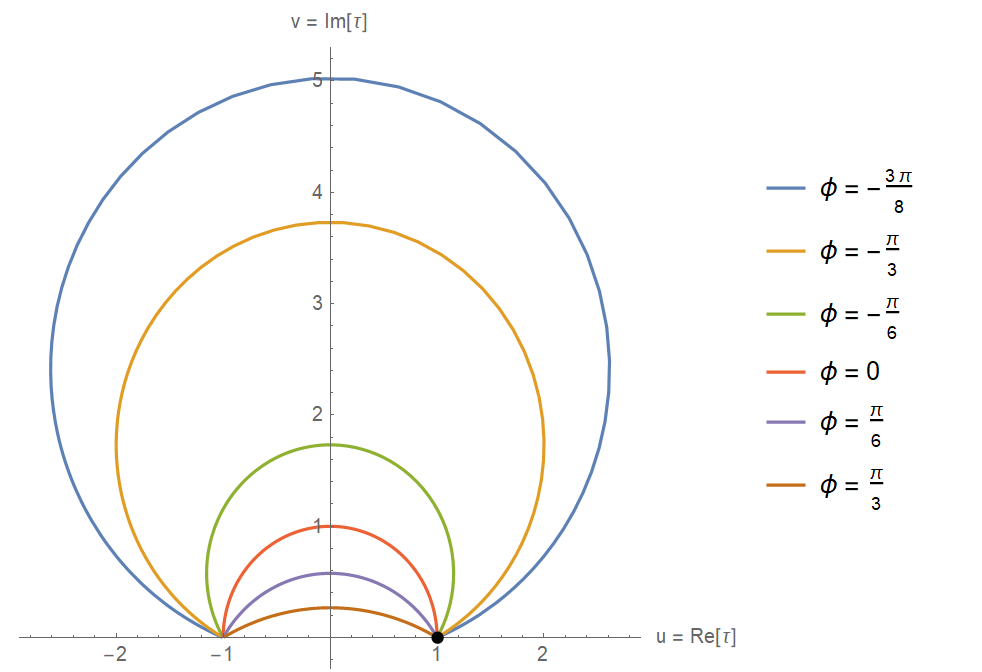}}{\fbox{\parbox{0.75\textwidth}{\centering Figure file \texttt{krugovi.png} not supplied.}}}
	\caption{Calabi-Eckmann structures $J_\tau$  deforming to same $J_\phi$}
	\label{fig:moduli}
\end{figure}

Now we visualize the contraction of Calabi-Eckmann structures.
The moduli space of Calabi-Eckmann complex structures  on $\mathfrak{so}(3)\oplus \mathfrak{so}(3)$ is $2$-dimensional and described in \eqref{eq:mtau} by parameter $\tau \in \CC,\, \operatorname{Im} (\tau) \neq 0.$
One can check that $\tau $ and $\frac{1}{\tau}$ and only them, represent the same structure on $\mathfrak{so}(3)\oplus \mathfrak{so}(3).$
Therefore, the moduli space of invariant complex structures on $\mathfrak{so}(3)\oplus \mathfrak{so}(3)$ is upper complex half plane.

Moreover, from the formula \eqref{eq:taut} we calculate $\tau _\phi (t)  = \frac{1}{\tau_ {\phi + \pi}(t)}.$ This corresponds to the fact that $J_\phi$ and $J_{\phi + \pi}$ are equivalent structures on $\e3 .$ 

The space of non-equivalent complex structures on $\e3$  is one-dimensional. In  Figure \ref{fig:moduli} Calabi-Eckmann structures $J_{\tau _\phi (t)}, t>0$, for fixed $\phi,$ are represented by circles passing through $-1$ and $1.$ They  contract to the same structure $J_\phi $ on $\e3$ as $t\to 0^+.$

\end{document}